\documentclass[a4paper,12pt]{amsart}

\usepackage[latin1]{inputenc}
\usepackage[T1]{fontenc}
\usepackage[english]{babel}

\usepackage{mathrsfs}
\usepackage{amscd}
\usepackage{amsfonts}
\usepackage{amsmath}
\usepackage{amssymb}
\usepackage{amstext}
\usepackage{amsthm}
\usepackage{amsbsy}

\usepackage{xspace}
\usepackage[all]{xy}
\usepackage{graphicx}
\usepackage{url}
\usepackage{latexsym}

\usepackage{booktabs} 
\usepackage{array} 
\usepackage{paralist} 
\usepackage{verbatim} 
\usepackage{subfig} 

\usepackage[textwidth=100pt,textsize=footnotesize,bordercolor=white,color=blue!30]{todonotes}
\usepackage{hyperref} 

\makeatletter
\newcommand*{\rom}[1]{\expandafter\@slowromancap\romannumeral #1@}
\makeatother

\theoremstyle{definition}

\newtheorem{fact}{Fact}

\newtheorem{thm}[fact]{Theorem}
\newtheorem{lemma}[fact]{Lemma}
\newtheorem{prop}[fact]{Proposition}
\newtheorem{corollary}[fact]{Corollary}
\newtheorem{defini}[fact]{Definition}
\newtheorem{rem}[fact]{Remark}

\newtheorem{question}[fact]{Question}
\newtheorem{remark}[fact]{Remark}

\title{Infinite computations with random oracles} 
\author{Merlin Carl and Philipp Schlicht} 
\date{\today} 

\begin{document} 
\maketitle 

\begin{abstract} 
We consider the following problem for various infinite time machines. If a real is computable relative to large set of oracles such as a set of full measure or just of positive measure, a comeager set, or a nonmeager Borel set, is it already computable? 
We show that the answer is independent from $ZFC$ for ordinal time machines ($OTM$s) with and without ordinal parameters and 
give a positive answer for most other machines. 
For instance, we consider 
infinite time Turing machines ($ITTM$s), unresetting and resetting infinite time register machines ($wITRM$s, $ITRM$s), and $\alpha$-Turing machines ($\alpha$-$TM$s) for countable admissible ordinals $\alpha$.


\end{abstract}

\tableofcontents

\section{Introduction}


If a real 
is Turing computable relative to all oracles in a set of positive measure, then it is Turing computable 
by a classical theorem of Sacks. 
Intuitively, this means that the use of random generators does not enrich the set of computable functions, not even when computability is weakened to computability with positive probability.
This insight 
refutes a possible objection against the Church-Turing-thesis, namely
that a computer 
could 
make randomized choices and thereby compute a function which is not computable by a purely deterministic device. 
The proof depends crucially on the compactness of 
halting Turing computations, i.e. the fact that only finitely many bits of an oracle are read in the course of a halting computation. 

Recently, the first author considered analogues of the Church-Turing-thesis for infinitary computations \cite{Ca2}. This naturally leads to the question whether a similar phenomenon can be observed concerning these
machine models. 
The situation is quite different for ordinal time Turing machines \cite{Ko3}, infinite time Turing machines \cite{HaLe}, unresetting infinite time register machines \cite{wITRM}, 
(resetting) infinite time register machines \cite{ITRM}, $\alpha$-Turing machines \cite{KoSe}, and ordinal time register machines  ($ORM$s) \cite{Ko1, Ko2}. 
All of these machines can consider each bit of a real oracle in the course of a halting computation. 
Nevertheless, the intuitive interpretation of computing relative to oracles in a set of positive measure as a 
randomized computation still makes sense. 

Hence we consider the following problem for each machine model. If a real is computable relative to large set of oracles such as a set of full measure or just of positive measure, a comeager set, or a nonmeager Borel set, is it already computable? 
We first show that this is independent from $ZFC$ for $OTM$s with and without ordinal parameters. 
We then give a positive answer for most other machines. 
For $ITTM$s, writability (eventual writability, accidental writability) in a nonmeager Borel set of oracles imply writability (eventual writability, accidental writability). For $ITRM$s of both kinds, computability in a set of positive measure or a nonmeager Borel set implies computability. For all (for unboundedly many) countable admissible ordinals  $\alpha$, computability by an $\alpha$-$TM$ from a nonmeager Borel set (a set of positive measure) of oracles implies computability.





\section{Ordinal Turing machines} 

Ordinal Turing machines ($OTM$s) can roughly be thought of as Turing machines with tape length and working time the class $Ord$ of ordinals. The machine state and tape content at limit times are obtained as the limit inferior of the earlier configurations. 
The definition and basic properties of $OTM$s can be found in \cite{OTM}.
We will call elements of ${}^{\omega}2$ 
\emph{reals}. 

\begin{defini} 
\begin{enumerate} 
\item A real $x\in {}^{\omega}2$ is \emph{$OTM$-computable} from a real $y\in {}^{\omega}2$ if there is an $OTM$ $P$ such that on input $y$, $P$ halts with output $x$, i.e. $P^y=x$. 
\item A set $A\subseteq {}^{\omega}2$ is \emph{$OTM$-computable} from a real $y$ if there is an $OTM$ $P$ such that for all $x\in {}^{\omega}2$,  $x\in A$ if and only if $P$ halts on input $x\oplus y$, i.e. $P^{x\oplus y}\downarrow$. 
\end{enumerate} 
\end{defini} 

It follows from the proof of \cite[Corollary 2]{SchSe} by application of the search algorithm that for any real $x$ such that $\{x\}$ is $OTM$-computable in $y$, or equivalently $\Delta^1_2$ in $y$, $x$ is $OTM$-computable 
in $y$. Conversely, if $x$ is $OTM$-computable from $y$, then $\{x\}$ is easily $OTM$-computable from $y$ by computing $x$ and comparing $x$ with the input. 
Since these two notions do not coincide for other machine models, computable reals are called \emph{writable} for most other machine models. 
We will say $OTM$-computable when we do not allow ordinal parameters, and $OTM$-computable with ordinal parameters otherwise. 

Let us first collect basic facts about ordinal time Turing machines and their halting times. Most of the results are folklore. 

\begin{defini} Let $\eta^x$ denote the supremum of halting times of $OTMs$ with oracle $x$. 
\end{defini} 

Note that there are gaps in the $OTM$ halting times. 

\begin{lemma} Suppose that $x$ is a real. 
\begin{enumerate} 
\item 
There are $\alpha<\beta$ such that $\beta$ is an $OTM$ halting time but $\alpha$ is not. 
\item All sets in $L_{\eta^x}$ are countable in $L_{\eta^x}$. 
\end{enumerate} 
\end{lemma} 

\begin{proof} We assume that $x=0$. We first show that for any $OTM$ halting time $\alpha$ of a program $P$, $L_{\alpha+\omega}-L_{\alpha}$ contains a real. The computation of $P$ is definable over $L_{\alpha}$ 
and hence in $L_{\alpha+1}$. Then $\alpha+1$ is minimal such that $P$ halts in $L_{\alpha+1}$. 
Then the hull of the empty set in $L_{\alpha+1}$ is $L_{\alpha+1}$. Hence there is a surjection from $\omega$ onto $L_{\alpha+1}$ definable over $L_{\alpha+1}$. Hence there is a real $x$ coding $L_{\alpha+1}$ in 
$L_{\alpha+2}$ and hence $x\in L_{\alpha+2}\setminus L_{\alpha}$. 

Suppose that $L_{\alpha+\omega}$ is the transitive collapse of a countable elementary substructure $M\prec L_{\omega_1+\omega}$ and $L_{\alpha}$ is the image of $M\cap L_{\omega_1}$. Then there are no reals in 
$L_{\alpha+\omega}\setminus L_{\alpha}$. Now suppose that $\gamma $ is least such that there are no reals in $L_{\gamma+\omega}\setminus L_{\gamma}$. Then $\gamma$ is not an $OTM$-halting time. We can search for 
$\gamma$ with an $OTM$ and can thus obtain a program with halting time $\geq \gamma$. 
\end{proof} 

This is analogous to $ITTMs$ \cite[Theorem 3.4]{HaLe}, but different from $ITRM$s, where the set of halting times is downwards closed \cite[Theorem 6]{ITRM2}. 


\begin{lemma} The following conditions are equivalent for reals $x,y$. 
\begin{enumerate} 
\item $x$ is $\Delta^1_2$ in $y$.  
\item $x$ is $OTM$-computable in the oracle $y$. 
\item $x\in L_{\eta^y}[y]$. 
\end{enumerate} 
\end{lemma} 

\begin{proof} Suppose that $x$ is $\Delta^1_2$ in $y$. Then $x$ is $OTM$-computable in the oracle $y$ by the proof of \cite[Corollary 2]{SchSe}. Since such a computation will last $<\eta^y$ steps, 
the computation and hence $x$ are in $L_{\eta^y}[y]$. Suppose that $x\in L_{\eta^y}[y]$. Then the $L$-least code for $L_{\beta}[y]$ is $\Delta^1_2$ in $y$. So $x$ is $\Delta^1_2$ in $y$. 
\end{proof} 

Thus $\eta^x$ is equal to the supremum 
of $\Delta^1_2$ wellorders in the parameter $x$ on $\omega$ by \cite[Corollary 6]{SchSe}. 

\begin{lemma} $\eta^x$ is an $x$-admissible limit of $x$-admissibles. 
\end{lemma} 

\begin{proof} To show that $\eta^x$ is $x$-admissible, it suffices to prove $\Delta_0$-collection in $L_{\eta^x}[x]$. Suppose that $y\in L_{\eta^x}[x]$ and that $R\subseteq y\times L_{\eta^x}[x]$ is
 $\Delta_0$-definable over $L_{\eta^x}[x]$ such that for every $u\in y$ there is some $v\in L_{\eta^y}[y]$ with $(u,v)\in R$. Let $P$ search on input $u\in y$ for the $L$-least $v$ with $(u,v)\in R$. 
The previous lemma implies that $\eta^y\leq\eta^x$. If we apply $P$ successively to all $z\in y$ and halt, then the halting time is some $\gamma<\eta^x$. Hence we can collect the witnesses in $L_{\gamma}[x]$. 

Suppose that $\alpha$ is the halting time of $P$ and $x=0$. Then the Skolem hull of the empty set in $L_{\alpha+1}$ is $L_{\alpha+1}$. Then there is a surjection from $\omega$ onto $L_{\alpha+1}$
 definable over $L_{\alpha+1}$, so $L_{\alpha}$ is countable in $L_{\eta}$. 
If $\eta$ is not a limit of admissibles, then $\eta=\omega_1^{CK,x}$ for some $x\in L_{\eta}$. Since we can compute $\omega_1^{CK,x}$ from $x$ with an $OTM$, this contradicts the definition of $\eta$. 
\end{proof} 

\begin{rem} $\eta^x$ is not $\Sigma_2$-$x$-admissible, since the function$f\colon \omega\rightarrow \eta^x$ which maps every halting program to its halting time is cofinal in $\eta^x$ and $\Delta_2$-definable over $L_{\eta^x}$. 
\end{rem} 

We will show that $\eta^x=\eta$ for Cohen reals $\eta$ over $L$, using the following lemma. 

\begin{lemma}{\label{OTMhaltingtimesup}}
Suppose that $x$ is a real. Let us call an ordinal $\alpha$ \emph{$\Sigma_{1}^{x}$-fixed} if and only if there exists a $\Sigma_{1}$-statement $\phi$ in the parameter $x$ such that $\alpha$ is minimal with the property that $L_{\alpha}[x]\models\phi(x)$.
Then $\eta^{x}$ is the supremum of the $\Sigma_{1}^{x}$-fixed ordinals.
\end{lemma}
\begin{proof}
First, we show that there is an $OTM$-halting time (in the oracle $x$) above every $\Sigma_{1}^{x}$-fixed ordinal: To see this, let $\alpha$ be $\Sigma_{1}^{x}$-fixed, say $\alpha$ is minimal such that $L_{\alpha}[x]\models\phi$, where $\phi$ is $\Sigma_{1}^{x}$.
 We will show below that there exists an $OTM$-program $P$ such that $P^{x}$
successively writes codes for all $L_{\alpha}[x]$ on the tape. Take such a program and check, after each step, whether the tape contains a code for some $L_{\beta}[x]$ such that $L_{\beta}[x]\models\phi$ and halt if this is the case. This program
obviously halts after at least $\alpha$ many steps, hence there is an $OTM$-halting time in the oracle $x$ which is at least $\alpha$.
For the other direction, take an $OTM$-program $P$ such that $P^{x}$ halts after $\alpha$ many steps. Hence, there exists $\beta>\alpha$ such that $L_{\beta}[x]$ contains the whole computation of $P^{x}$. This $\beta$ is minimal such that $L_{\beta}$ believes
that $P^{x}$ halts, i.e. that the computation of $P^{x}$ exists, which is a $\Sigma_{1}^{x}$-statement. Hence $\beta>\alpha$ is $\Sigma_{1}^{x}$-fixed. 
Consequently, the suprema coincide.
\end{proof} 

\begin{prop} If $x$ is Cohen generic over $L$, then $\eta^x=\eta$. 
\end{prop} 

\begin{proof} Suppose that $x$ is Cohen generic over $L$ and $P^x$ halts at time $\gamma$. Let $\varphi(y,\alpha)$ state that $P^y$ halts at time $\alpha$. 
Suppose that $\dot{x}$ is the canonical name for the Cohen real and that $p\Vdash \varphi(\dot{x},\gamma)$. Since $\varphi$ is $\Sigma_1$, the existence of some $\alpha$ with $p\Vdash \varphi(\dot{x},\alpha)$ is $\Sigma_1$, 
so this holds in $L_{\eta}$ by Lemma \ref{OTMhaltingtimesup}. 
So there is some $\alpha<\eta$ with $p\Vdash \varphi(\dot{x},\alpha)$. Then $P^x$ halts at time $\alpha<\eta$ in $L[x]$, so $\alpha=\gamma<\eta$. 
\end{proof}

\subsection{Computations without parameters} 


Natural numbers as oracles do not change Turing computability. Thus there are at least two natural generalizations of Turing computability to computations of ordinal length, with and without ordinal parameters.
 We first consider machines without ordinal parameters. 

We first show that 
in $L$ there is a non-computable real $x$ which is computable relative to all oracles in a set of measure $1$. 
Let us say that a set $c$ of ordinals \emph{codes} a transitive set $x$ if there is some 
$\gamma\in Ord$ and a bijection $f:\gamma\rightarrow x$ such that $c=\{p(\alpha,\beta)\mid  \alpha,\beta<\gamma,\  f(\alpha)\in f(\beta)\}$, where $p\colon Ord \times Ord \rightarrow Ord$ denotes G\"odel pairing. 


\begin{lemma}{\label{enumerateL}} 
\begin{enumerate} 
\item There is an $OTM$ program $P$ 
such that for every $\alpha\in Ord$, there is an ordinal $\beta$ such that the tape contents 
at time $\beta$ is the 
characteristic function of a code for $L_{\alpha}$. 
\item There is an $OTM$ program $Q$ which stops with output $1$ if and only if the tape contents at the starting time is a code for some $L_{\alpha}$, and $Q$ stops with output $0$ otherwise. 
\item There is an $OTM$ program $R$ which for an arbitrary real $x$ in the oracle, stops with output $1$ if and only if the tape contents at the starting time is a code for some $L_{\alpha}$ with $x\in L_{\alpha}$. 
\end{enumerate} 
\end{lemma}
\begin{proof}
Note that $x\subseteq Ord$ is $OTM$-computable from finitely many ordinal parameters if and only if $x\in L$ \cite{OTM}. 
The program $P$ is obtained as follows. 
We enumerate all tupes $(m,\alpha_{0},....,\alpha_{n})$ with $m\in \omega$ and ordinals $\alpha_1,...,\alpha_n$. Let the $m^{th}$ $OTM$ program $P_m$ run for $\alpha_0$ many steps in the parameter $(\alpha_{0},...,\alpha_{n})$. 
This generates codes for all elements of $L$, in particular for $L_{\alpha}$ for all $\alpha\in Ord$. 

For the second claim, note that bounded truth predicates can be computed by an $OTM$ \cite{Ko1}. 
The wellfoundedness of the tape content can be tested 
by an exhaustive search. We can then check the sentence $\exists\alpha\in Ord\ V=L_{\alpha}$ by 
evaluating the bounded truth predicate. 

For the third claim, we can check whether the tape contents codes some $L_{\alpha}$ by the second claim. Whether or not $x\in L_{\alpha}$ can be checked by identifying representatives for all elements of $\omega$ in the coded structure and then
checking for each element of the coded structure whether it is equal to $x$. 
Whether some $\delta\in Ord$ codes $n\in\omega$ can be checked as follows: If $n=0$, then one runs through the code to check whether $\delta$ 
has any predecessors, e.g. whether $p(\gamma,\delta)$ belongs to the code for some $\gamma$. Then, recursively, a code for $n+1$ can be identified as having exactly the codes for $0,1,...,n$ as its predecessors. 
\end{proof}

\begin{thm}{\label{InL}} 
Suppose that $V=L$. 
There is a real $x$ and a 
co-countable set $A\subseteq {}^{\omega}2$
such that $x$ is $OTM$-computable without ordinal parameters from every $y\in A$, but $x$ is 
not $OTM$-computable without parameters. 
\end{thm} 
\begin{proof} 
Since there are only countably many $OTM$ programs, there are only countably many halting times without parameters, all of which are countable due to condensation in $L$. Let $\alpha<\omega_1$ denote their supremum. 
Let $A={}^{\omega}2\setminus L_{\alpha}$ and suppose that $x$ is the $<_{L}$-least real coding a well-ordering of order type $\alpha$. 
We claim that $x$ is $OTM$-computable without parameters relative to any $y\in A$. 
To see this, suppose that $P$ is a diverging $OTM$ program which writes $L_{\beta}$ on the tape for all $\beta\in Ord$ as in Lemma \ref{enumerateL}. 
We wait for the least $\beta\in Ord$ with $y\in L_{\beta}$. 
Then $x\in L_{\beta}$ and hence $\alpha<\beta$. 
We then write a sequence of $\beta$ many $1$s on the tape, succeeded by $0$s. This allows us to solve the halting problem for parameter free $OTM$s as follows. 
Whenever a program runs for $\beta$ many steps, it cannot halt, since $\beta>\alpha$. 
We compute the supremum $\alpha$ of the halting times and then search $L_{\beta}$ for the 
$L$-least code $x$ for $\alpha$. However, $x$ itself is not $OTM$-computable, as it would allow us to write a sequence of $\alpha$ many $1$s on the tape succeeded by $0$s, which allows a solution of the halting problem for parameter free $OTM$s. 
\end{proof} 

\begin{corollary}{\label{MoreInL}} Assume that $V=L$. 
\begin{enumerate} 
\item Let $h$ be a real coding the halting problem for parameter-free $OTM$s. Then $h$ is $OTM$-computable from every non-$OTM$-computable real $x$.
\item For all reals $x$ and $y$, $x$ is $OTM$-computable from $y$ or $y$ is $OTM$-computable from $x$.
\end{enumerate} 
\end{corollary} 
\begin{proof}
The first claim follows from the previous proof. 
For the second claim, let $\alpha$ and $\beta$ be minimal such that $x\in L_{\alpha+1}$ and $y\in L_{\beta+1}$. Assume without loss of generality that $\beta\geq\alpha$.
Given $y$, we can, using the strategy from the proof of Theorem \ref{InL}, compute the $<_{L}$-minimal real $r$ coding
an $L_{\beta+1}$. As $x\in L_{\beta+1}$, it must be coded by some fixed natural number $n$ in $r$ which can be given to our program in advance.
It is now easy to compute $x$ from $r$. Thus $x$ is computable from $y$.
\end{proof}



It is also consistent that there is no such real $x$. 

\begin{thm}{\label{OhneParameter}} 
\begin{enumerate} 
\item Suppose that for every $x\in {}^{\omega}2$, the set of random reals over $L[x]$ has measure $1$. If $A$ has positive Lebesgue measure and $x\in {}^{\omega}2$ is $OTM$-computable without ordinal parameters from every $y\in A$, 
then $x$ is $OTM$-computable without ordinal parameters. 
\item Suppose that for every $x\in {}^{\omega}2$, the set of Cohen reals over $L[x]$ is comeager. If $A$ is nonmeager with the property of Baire and $x\in {}^{\omega}2$ is $OTM$-computable without ordinal parameters from 
every $y\in A$, then $x$ is $OTM$-computable without ordinal parameters. 
\end{enumerate} 
\end{thm} 

\begin{proof} Suppose that for every $x\in {}^{\omega}2$, the set of random reals over $L[x]$ has measure $1$, and that $x\in {}^{\omega}2$ is $OTM$-computable without ordinal parameters from every $y\in A$. 
If $B\subseteq {}^{\omega}2\times {}^{\omega}2$ is $\Sigma^1_2$ and $q\in \mathbb{Q}$, then the set $\{x\in {}^{\omega}\omega \mid \mu(B_x)> q\}$ is $\Sigma^1_2$, by the proof of \cite[Theorem 2.2.3]{Ke}. Note that as stated, this proof 
uses projective determinacy, but Lebesgue measurability of $\bf\Sigma^1_2$ sets is sufficient for the application of \cite[Corollary 2.2.2]{Ke}. 

Now suppose that $\mu(A)>0$ and for every $y\in A$, there is an $OTM$ $P$ such that $P^y$ computes $x$. Then $\{y\mid P^y=x\}$ is provably $\Delta^1_2$ and hence measurable by \cite[Exercise 14.4]{Kanamori}. 
Since there are countably many programs, $\mu(\{y\mid P^y=x\})>0$ for some program $P$. There is a basic open set $U$ such that the relative measure of $\{y\mid P^y=x\}$ in $U$ is $>0.5$ by the Lebesgue density theorem. 

We can assume without loss of generality that $U={}^{\omega}2$. 
Then $\{x\}=\{y\mid  \mu(\{z\mid y=P^z\})>0.5\}$, so $\{x\}$ is $\Sigma^1_2$ and thus easily $\Delta^1_2$ . Note that a set $A$ of reals is $OTM$-computable iff it is $\Delta^{1}_{2}$ by Corollary $3.11$ of \cite{Se}. 
It follows from the discussion in the beginning of this section that $x$ is $OTM$-computable. 

The proof of the second claim is analogous. 
\end{proof} 

\begin{remark} \label{remark pfOTMindependence} 
\begin{enumerate} 
\item The statement that for every real $x$, the set of random reals over $L[x]$ has measure $1$ is equivalent to the statement that every $\bf\Sigma^1_2$ set is Lebesgue measurable \cite[Theorem 4.4]{Ik}. 
This follows from $\omega_1^{L[x]}<\omega_1$ for all reals $x$ by \cite[Corollary 14.3]{Kanamori}. 
\item The statement that for every real $x$, the set of Cohen reals over $L[x]$ is comeager is equivalent to the statement that every $\bf\Sigma^1_2$ set hast the property of Baire \cite[Theorem 4.4]{Ik}. 
This follows from $\omega_1^{L[x]}<\omega_1$ for all reals $x$ by \cite[Corollary 14.3]{Kanamori}. 
\end{enumerate} 
\end{remark} 

It follows from Theorems \ref{InL} and Theorem \ref{OhneParameter} that 
$ZFC$ does not decide whether there is a real $x$ which is not $OTM$-computable and a Borel set $A\subseteq{}^{\omega}2$ which is nonmeager or has positive measure such that $x$ is $OTM$-computable from every element of $A$. 





\subsection{Computations with real parameters} 

We will see below that, for most machine concepts of transfinite computability, computability with positive probability relative to a random oracle does not exceed plain computability. 
Since parameter-free $OTM$-computation
provides a natural formalization of the intuitive idea of a transfinite construction procedure, this 
intrinsically motivates the consideration of the statement that every real $x$ which is $OTM$-computable relative to all reals $y$
from some set $A$ with $\mu(A)>0$ is $OTM$-computable in the empty oracle. We will abbreviate this axiom 
by $Z(0)$. 
Similarly, for an arbitrary real $x$, we denote by $Z(x)$ the statement that every real which is $OTM$-computable relative to all reals $x\oplus y$ for all $y\in A$ with $\mu(A)>0$ is $OTM$-computable in the oracle $x$. 
Intuitively $\neg Z(x)$ means that $x$ contains a way of extracting new information from randomness, so we call a real $x$ with $\neg Z(x)$ an extracting real. The same intuition 
motivates the consideration of the statement $Z$ 
that no extracting reals exist. 

We easily obtain similar results as above for computability relative to real oracles.

\begin{prop}
$Z$ is independent from $ZFC$.
\end{prop}
\begin{proof}
The failure of $Z(0)$ implies the failure of $Z$, so $Z$ fails in $L$ by Theorem \ref{InL}. 
On the other hand, the proof of Theorem \ref{OhneParameter} shows that $Z$ holds if for every real $x$, the set of random reals over $L[x]$ has measure $1$. 
This is consistent by Remark \ref{remark pfOTMindependence}. 
\end{proof}

As a consequence of $Z$, 
the universe $V$ cannot be too close to $L$. 

\begin{prop}
$ZFC+Z$ 
implies that $V\neq L[x]$ for all reals $x$. 
\end{prop}
\begin{proof}
It suffices to show that $Z(x)$ fails in $L[x]$. To see this, we follow the proof of Theorem \ref{InL} above and obtain a (non-halting) $OTM$-program $P$ such
that $P^{x}$ writes $L[x]$ on the tape. Let $\sigma^{x}$ be the supremum of the halting times of $OTM$s in the oracle $x$. Then for all but countably many reals $y$ in $L[x]$, the smallest $\beta$ such that
$y\in L_{\beta+1}[x]$ will be larger than $\sigma^{x}$ by a standard fine structural argument (note that, as $x$ is just a real, condensation holds in the $L[x]$-hierarchy). Now suppose that $(P_{i})_{i\in\omega}$ is a computable enumeration of the $OTM$-programs and proceed as follows. Given $x\oplus y$ in the oracle, use $P^{x}$ to enumerate $L[x]$ until some $L_{\beta}[x]$ is found that contains $y$. Then compute a
set $H^{x}_{y}\subseteq\omega$ by letting $P_{i}^{x}$ run for $\beta$ many steps and outputting $1$ if $P_{i}^{x}$ stops before time $\beta$ and $0$, otherwise. When $\beta>\sigma^{x}$, then $H^{x}_{y}$ will
just be the halting number for $OTM$-programs in the oracle $x$, which is not $OTM$-computable in the oracle $x$. As we observed above that the $\beta$ we find will be $>\sigma^{x}$ for all but countably many
 reals $y$, this procedure computes the halting number for $OTM$s in the oracle $x$ relative to $x\oplus y$ for all but countable many $y$, and hence a set of $y$ of measure $1$, contradicting $Z(x)$.
\end{proof}

Since $x^{\#}$ exists in $L[x^{\#}]$, the existence of $x^{\#}$ does not imply $Z$. However, the existence of $x^{\#}$ for all reals $x$ implies that $\omega_1^{L[x]}<\omega_1$ for all reals $x$ and hence $Z$ by Remark  \ref{remark pfOTMindependence} . 

\begin{question} 
\begin{enumerate} 
\item Is it consistent that $Z(0)$ holds while $Z$ fails? 
\item Is it consistent that $Z(0)$ holds in $L[x]$ for some real $x$? 
\item Does $Z$ imply that there are random reals over $L$? 
\end{enumerate} 
\end{question}

\subsection{Computations with ordinal parameters} 

In analogy with Turing machines, where arbitrary natural numbers are allowed as oracles, we can allow ordinals as oracles as in \cite{KoSe}. 
For this type of computations, a real $x$ is computable from a real $y$ if and only if there exist an $OTM$-program $P$ and finitely many ordinals $\alpha_{0},...,\alpha_{n}$ 
such that $P$ eventually stops with $x$ written on the tape, when run in the oracle $y$ with parameters $\alpha_{0},...,\alpha_{n}$. The computability strength corresponds to constructibility. 
We obtain the following fact by a straightforward relativization of the proof of \cite{KoSe}. 

\begin{lemma}
A real $x$ is $OTM$-computable from $y$ with ordinal parameters if and only if $x\in L[y]$.
\end{lemma}

We aim to characterize the models of set theory where random oracles cannot add information, i.e. where $OTM$-computability with ordinal parameters from all oracles in a set of positive measure implies 
$OTM$-computability with ordinal parameters in the empty oracle. 
Trivially, $L$ has this property. 
Note that if ${}^{\omega}2\not\subseteq L$ and the set of constructible reals is measurable, then it has measure $0$. This follows from the fact that we can partition $\omega$ into a constructible sequence of 
disjoint infinite sets and translate ${}^{\omega}2\cap L$ by some $a\in {}^{\omega}2\setminus L$ separately on each set. 

If ${}^{\omega}2\cap L$ is not measurable, then every set of reals of positive measure contains a real in $L$ and this real is $OTM$-computable with ordinal parameters. 
Many forcings such as 
random forcing and Sacks forcing preserve outer measure, so that in the generic extension the set of ground model reals is not measurable. Such extensions of $L$ also have the required property. 

We now consider the case that ${}^{\omega}2$ has measure $0$. 
Note that the statement that a code $c\in {}^{\omega}2$ for a Borel subset of ${}^{\omega}2$ codes a measure $1$ set is absolute between transitive models of $ZFC$ containing $c$ by \cite[Lemma 26.1]{Je}. 
This implies that for every generic filter $g$ over $M$ and every random real $x$ over $M[g]$, $x$ is random over $M$. 
The random reals appearing in a two-step iteration of random forcing are not mutually random generic by \cite[Lemma 3.2.8, Theorem 3.2.11]{BJ}. 
However, the next lemma is sufficient for our application. 

\begin{lemma}\label{successive random reals} Suppose that $M$ is a model of $ZFC$. Suppose that $x$ is random over $M$ and $y$ is random over $M[x]$. Then $M[x]\cap M[y]
=M$. 
\end{lemma} 

\begin{proof} Let $\mathbb{P}$ denote random forcing and $\dot{\mathbb{P}}$ a $\mathbb{P}$-name for random forcing. Note that $\mathbb{P}*\dot{\mathbb{P}}$ is forcing equivalent to $\mathbb{P}$ \cite[Lemma 3.2.8]{BJ}. 
Let $\dot{x},\dot{y}$ be names for the random reals added by $\mathbb{P}*\dot{\mathbb{P}}$. 

We claim that there is a condition $(p,\dot{q})\in\mathbb{P}*\dot{\mathbb{P}}$ with $(p,\dot{q})\Vdash_{\mathbb{P}*\dot{\mathbb{P}}}\dot{M}[\dot{x}]\cap \dot{M}[\dot{y}]=\dot{M}$, where $\dot{M}$ is a name for the ground model $M$.  
Otherwise $1_{\mathbb{P}}\Vdash_{\mathbb{P}*\dot{\mathbb{P}}}\dot{M}[\dot{x}]\cap \dot{M}[\dot{y}]\neq \dot{M}$. 
Let $\kappa=(2^{\omega})^M$. 
Suppose that $g$ is generic over $M$ for  a finite support product $\mathbb{P}$ of $(\kappa^+)^M$ random forcings. Note that random forcing is $\sigma$-linked by \cite[Lemma 3.1.1]{BJ} and hence Knaster. Then $\mathbb{P}$ is Knaster and hence c.c.c. by \cite[Corollary 15.16]{Je}. 

Let $(x_{\alpha})_{\alpha<(\kappa^+)^M}$ denote the sequence of random reals added by $g$. Suppose that $y$ is random over $M[g]$. 
Then $y$ is random over $M$ and over $M[x_{\alpha}]$ for all $\alpha<(\kappa^+)^M$, so $M[x_{\alpha},y]$ is a $\mathbb{P}*\dot{\mathbb{P}}$-generic extension of $M$. 
Hence there is some $y_{\alpha}\in (M[x_{\alpha}]\cap M[y])\setminus M$ for each $\alpha<(\kappa^+)^{M}$. 
Then $x_{\alpha}, x_{\beta}$ are mutually generic for all $\alpha\neq\beta$. 
This implies $M[x_{\alpha}]\cap M[x_{\beta}]=M$ by a similar argument as in Lemma \ref{mutgensect} below. 
Then $y_{\alpha}\neq y_{\beta}$ for $\alpha\neq \beta$ and 
hence $|2^{\omega}|^{M[G]}=|(2^{\omega})^{M[y]}|^{M[G]}$. But $|2^{\omega}|^{M[G]}=(\kappa^+)^{M[G]}=(\kappa^+)^M$ and $|(2^{\omega})^{M[y]}|^{M[G]}=|2^{\omega}|^{M[y]}=|2^{\omega}|^M=\kappa$, since random forcing and $\mathbb{P}$ are c.c.c. 

Suppose that $(p,\dot{q})\Vdash_{\mathbb{P}*\dot{\mathbb{P}}}\dot{M}[\dot{x}]\cap \dot{M}[\dot{y}]=\dot{M}$. 
It follows from the isomorphism theorem for Borel measures \cite[Theorem 17.41]{Ke2} that 
for every condition $r\in \mathbb{P}$, random forcing below $r$ is forcing equivalent to $\mathbb{P}$, i.e. the Boolean completions are isomorphic. 
Thus for an arbitrary random real $x$ over $V$ not necessarily below $p$, there is some condition $r\in \dot{\mathbb{P}}^{x}$ with $r\Vdash^{M[x]}_{\mathbb{P}} \dot{M}[\check{x}]\cap \dot{M}[\dot{y}]=M$. 
Then $M[x]\cap M[y]=M$ for an arbitrary random real $y$ over $M[x]$ by the same argument for $\dot{\mathbb{P}}^x$. 
\end{proof} 

\begin{thm}\label{reals constructible from many reals} 
\begin{enumerate} 
\item Suppose that for every real $x$, there is a random real over $L[x]$. If $A$ has positive measure and $x\in {}^{\omega}2$ is constructible from each $y\in A$, then $x\in L$. 
\item Suppose that for every real $x$, there is a Cohen real over $L[x]$. If $A$ is a nonmeager Borel set and $x\in {}^{\omega}2$ is constructible from each $y\in A$, then $x\in L$. 
\end{enumerate} 
\end{thm} 

\begin{proof} Since $A$ has a Borel subset with the same measure, we can assume that $A$ is Borel. Suppose that $a$ is a Borel code for $A$. 
Note that a real $y$ is random over a model $M$ if and only if $y$ is in every measure one Borel set coded in $M$. Let $y$ be random over $L[a]$ below $A$ and $z$ random over $L[a][y]$ below $A$. Such reals $y,z$ exist since 
random forcing below the condition $A$ is 
forcing equivalent to random forcing. 
Then $y,z\in A$. Moreover $y$ is random over $L$ and $z$ is random over $L[y]$ by the discussion before the previous lemma. 
Since $x$ is constructible from $y$ and from $z$ by our assumption, $x\in L$ by the previous lemma applied to $V=L$. 

The argument for Cohen forcing is similar. 
\end{proof} 

\begin{remark} 
\begin{enumerate} 	
\item After forcing with a finite support iteration of length $\omega_1$ of random forcings, there is a random real over $L[x]$ for every real $x$. The statement that for every real $x$, there is a random real over $L[x]$ is equivalent to the statement that every $\bf\Delta^1_2$ set is Lebesgue measurable \cite[Theorem 4.3]{Ik}. 
\item After forcing with a product of $\omega_1$ Cohen forcings, there is a Cohen real over $L[x]$ for every real $x$. 
The statement that for every real $x$, there is a Cohen real over $L[x]$ is equivalent to the statement that every $\bf\Delta^1_2$ set has the property of Baire \cite[Theorem 4.3]{Ik}. 
\end{enumerate} 
\end{remark} 

There is a forcing extension of $L$ such that there is a nonconstructible real $x$ which is constructible from all elements of a measure $1$ set 
\cite[Section 3]{JuShe}. 

\begin{thm}[Judah-Shelah]\label{Goldstern forcing} There is a forcing $\mathbb{P}$ in $L$ such that in any $\mathbb{P}$-generic extension of $L$, there is a measure one set $A$ such that every $x\in A$ can be constructed from every $y\in A$, but $A$ contains no constructible real. 
\end{thm} 

\begin{proof} Blass-Shelah forcing has this property \cite[Section 3]{JuShe}. We include a much shorter proof via a simplification of a forcing of Martin Goldstern, whom we thank for allowing us to include this. 
We define a forcing $\mathbb{P}$ with the property that every new real constructs the generic real, i.e. the forcing is minimal for reals, and the set of ground model reals has measure $0$. 
Suppose that $(a_n)_{n\in\omega}$ is a strictly increasing sequence of natural numbers with $a_{n+1}-a_n\geq n$. Let $I_n=[a_n,a_{n+1})$. 
The forcing $\mathbb{P}$ consists of trees $t$ whose nodes of $t$ are of the form $(C_0,...,C_n)$ with $C_i= 2^{I_i}\setminus\{t_i\}$ for some $t_i\in 2^{I_i}$. Then $\mu(C_i)\geq 1-\frac{1}{2^i}$. Every splitting node $(C_0,...,C_n)$
 splits into $(C_0,...,C_{n+1})$ for all such $C_{n+1}$. The trees have no end nodes and cofinally many splitting nodes. The conditions are ordered by reverse inclusion. 

Suppose that $(C_n)_{n\in\omega}$ is $\mathbb{P}$-generic over $V$. 
Then $\mu(\{x\mid \forall^{\infty} n\ x\upharpoonright I_n \in C_n\})=1$.  
Let $X=\{x\mid \exists^{\infty} n\ x\upharpoonright I_n\notin C_n\}$. Then $\mu(X)=0$. 
Suppose that $x\in {}^{\omega}2\cap V$. Then for any $t\in \mathbb{P}$ with the stem $(D_0,...,D_n)$, we can find some $s\leq t$ by choosing $D_{n+1}$ with $x\upharpoonright I_{n+1}\notin D_{n+1}$, hence $(D_0,...,D_{n+1})$ forces 
that $x\upharpoonright I_{n+1}\notin \dot{C}_{n+1}$, where $\dot{C}_{n+1}$ is a name for $C_{n+1}$. This implies that $x\in X$. Thus $\mu({}^{\omega}2 \cap V)=0$. 

We claim that $\mathbb{P}$ has the pure decision property, i.e. given any $s\in \mathbb{P}$ and any sentence $\varphi$, there is some $t\leq s$ with the same stem as $s$ which decides $\varphi$. As for Sacks forcing, we enumerate the 
direct successors of the stem $t_0$ of $t$ as $u_0,...,u_n$ and choose trees $t^{i}\leq t/u_i=\{r\in t\mid u\subseteq u_i$ or $u_i\subseteq r\}$ deciding $\varphi$. Then $s=\bigcup_{i\leq n} t^{i}$ has the stem $t_0$ and decides $\varphi$. 

If $t$ forces that $\dot{x}$ is a name for a new real, we can build a subtree $s\leq t$ using the pure decision property such that at every splitting node $p$ in $s$, the parts of $\dot{x}$ decided by $s/q$ for direct successors of $p$ 
 are incompatible. This can easily be done by considering all pairs of direct successors, since the trees are finitely splitting. Then the generic real $y$ is the unique branch in $s$ which is compatible with $\dot{x}^y$ and hence is constructible from $\dot{x}^y$. 
\end{proof} 

It is independent from $ZFC$ whether there are a real $x$ and a set $X\subseteq {}^{\omega}2$ 
of positive measure such that $x$ is $OTM$-computable with parameters from each element of $X$, by Theorem \ref{reals constructible from many reals} and Theorem \ref{Goldstern forcing}. 
The same statement, but with sets of positive measure replaced by nonmeager Borel sets, is independent from $ZFC$ by Theorem \ref{reals constructible from many reals} and the following property of Laver forcing. 

\begin{thm}[Gray] Laver forcing adds a minimal real such that the set of ground model reals is meager. 
\end{thm} 

\begin{proof} Laver forcing is minimal \cite{Gr}. Since a Laver real dominates the ground model reals \cite[Theorem 7.3.28]{BJ}, the set of ground model reals is meager in the generic extension. 
\end{proof} 

\begin{remark} The results in this section hold verbatim for Ordinal Register Machines ($ORM$s) (introduced in \cite{Ko1}) which are identical to $OTM$s in computational strength with and without 
 ordinal parameters. This is shown in \cite{Ko2} in the case with parameters. We leave out the proof for the case without parameters, which is not hard to obtain, but technical and not very informative. 
\end{remark}

Note that in the situation of Theorem \ref{Goldstern forcing}, for any new real $x$, we can search through all $\mathbb{P}$-names $\dot{x}$ in the ground model $M$ and thin out trees as in the proof of Theorem \ref{Goldstern forcing}. For each such tree $t$, we compute the unique branch $y$ with $\dot{x}^y=x$, if it exists, and check whether it is $\mathbb{P}$-generic over $L$. Thus we have an $OTM$ program which computes a $\mathbb{P}$-generic real over $L$ from each new real. 

\begin{question} Is it consistent that there is a nonconstructible real $x$ and a Borel set $A$ of measure $1$ such that $x$ is $OTM$-computable without parameters from every $y\in A$? 
\end{question} 

More generally, we ask which combinations of the following statements are consistent (with $\mu({}^{\omega}2\cap L)=0$). If $A$ is a Borel set of positive measure (measure $1$) and $x$ is $OTM$-computable (with ordinal parameters) from each $y\in A$, then $x$ is $OTM$-computable (with ordinal parameters). 


%

\section{Infinite time Turing machines} 

$ITTM$s are the historically first machine model of transfinite computations. Roughly speaking, an $ITTM$ is a classical Turing machine with transfinite ordinal running time: Whenever the time reaches 
a limit ordinal, the tape content at each cell is the limit inferior of the earlier contents and the machine assumes a special limit state. The definitions of  $ITTM$s, writability, eventual writability and accidental writability can be found in \cite{HaLe}. 

In this section, we will show that every real $x$ which is writable (eventually writable, accidentally writable) from every real in a nonmeager Borel set is already writable (eventually writable, accidentally writable, respectively). 
%
%
%
The proofs use Cohen forcing over $L_{\alpha}$. 
A similar argument using a ranked forcing language can be found in \cite[Theorem 3.1]{We3}. 
In ongoing work, we are attempting to use a similar strategy for random forcing instead of Cohen forcing, which would lead to the analogous result for positive Lebesgue measure. The difficulty is that random forcing in $L_{\alpha}$ is a proper class. 

\begin{defini} Suppose that $y$ is a real. Let $\lambda^{y}$ ($\zeta^y$, $\Sigma^y$) denotes the supremum of the ordinals writable (eventually writable, accidentally writable) in the oracle $y$. Let $\lambda=\lambda^0, \zeta=\zeta^0$, $\Sigma=\Sigma^0$. 
\end{defini} 

Welch characterized the writable (eventually writable, accidentally writable) reals \cite{We1}. 

\begin{thm}[Welch]{\label{ITTMchar}} For every real $x$, the reals writable (eventually writable, accidentally writable) in the oracle $x$ are exactly those in $L_{\lambda^{x}}[x]$ ($L_{\zeta^{x}}[x]$, $L_{\Sigma^{x}}[x]$). 
\end{thm}

Note that $\zeta$ is $\Sigma_2$-admissible and $\Sigma$ is a limit of $\Sigma_2$-admissibles \cite[Lemma 7, p. 19]{We2}, but $\Sigma$ is not admissible \cite[Fact 2]{We2}. 
Moreover $\lambda$ is an admissible limit of admissibles by \cite[Fact 2.2, p. 11]{We2}. 
Since adding an oracle can only increase the supremum of the writable (eventually writable, accidentally writable) ordinals, we have $\lambda\leq \lambda^x$, $\zeta\leq \zeta^x$, and $\Sigma\leq \Sigma^x$ for all reals $x$. 

Our goal is to show that $\lambda^x=\lambda$, $\zeta^x=\zeta$, and $\Sigma^x=\Sigma$ for Cohen generic reals $x$ over $L_{\Sigma+1}$, using the following characterization. 
The proof of the unrelativized version can be found in \cite[Theorem 2.1, Theorem 2.3]{We1}. The relativized version is discussed in the proof of \cite[Lemma 2.4]{We1}. 

\begin{thm}[Welch]{\label{lazesithm}} 
Suppose that $x$ is a real. Then $(\zeta^{x},\Sigma^{x})$ is the lexically minimal pair of ordinals such that $L_{\zeta^{x}}[x]\prec_{\Sigma_{2}}L_{\Sigma^{x}}[x]$. Moreover, $\lambda^{x}$ is minimal 
with the property that $L_{\lambda^{x}}[x]\prec_{\Sigma_{1}}L_{\zeta^{x}}[x]$.
\end{thm}


Although we only need to force over $L_{\alpha}$ where $\alpha$ is admissible or a limit of admissibles, let us phrase the results in a stronger form. 
Mathias \cite{Ma} developed set forcing over models of a weak fragment $PROV$ of $ZFC$ such that the transitive models of $PROV$, the \emph{provident sets}, are 
the transitive sets closed under functions defined by recursion along rudimentary functions and containing $\omega$. 
The definitions and basic facts about rudimentary functions and provident sets can be found in \cite{Ma, MaBo}. For example, $L_{\alpha}$ is provident if and only if $\alpha$ is an infinite indecomposable ordinal.
We would like to thank Adrian Mathias for discussions on this topic. 

As usual, if $\mathbb{P}\subseteq L_{\alpha}$ is a partial order and $G\subseteq \mathbb{P}$ is a filter, let $L_{\alpha}[G]=\{\sigma^G\mid \sigma\in L_{\alpha}\}$ denote the \emph{generic extension} of $L_{\alpha}$ by $G$. 
Let $L_{\alpha}^x$ denote $L_{\alpha}$ built relative to the language $\{\in,x\}$, where $x$ is a real. 
If $L_{\alpha}$ is provident and $x$ is Cohen generic over $L_{\alpha}$, then $L_{\alpha}[x]=L_{\alpha}^x$ by \cite[Section 9]{Ma}. 


\begin{lemma}{\label{mutgensect}} Suppose that $L_{\alpha}$ is provident, $\mathbb{P}, \mathbb{Q}\in L_{\alpha}$ are forcings, and $G\times H$ is $\mathbb{P}\times \mathbb{Q}$-generic over $L_{\alpha}$. Then $L_{\alpha}[G]\cap L_{\alpha}[H]= L_{\alpha}$. 
\end{lemma} 

\begin{proof} 
The forcing relation for atomic formulas is definable by a rudimentary recursion over provident sets by \cite[Section 2]{Ma}, and the forcing relation for $\Delta_0$ formulas is rudimentary in the forcing relation for atomic formulas \cite[Section 3]{Ma}. 
Hence $\{(p,q)\in \mathbb{P}\times\mathbb{Q}\mid p\Vdash \check{q}\in \sigma\}\in L_{\alpha}$ for any $\mathbb{P}$-name $\sigma\in L_{\alpha}$. 
Thus a filter $F\subseteq \mathbb{P}\times\mathbb{P}$ is $\mathbb{P}\times\mathbb{P}$-generic over $L_{\alpha}$ if and only if there is a $\mathbb{P}$-generic filter $G$ over $L_{\alpha}$ and a $\mathbb{P}$-generic filter $H$ over $L_{\alpha}[G]$ with $F=G\times H$, by the proof of \cite[Lemma 15.9]{Je}. 

Let $\dot{G}$, $\dot{H}$ denote the canonical names for $G,H$. Suppose that $x$ is of minimal rank with $x\in L_{\alpha}[G]\cap L_{\alpha}[H]$ and $x\notin L_{\alpha}$. Suppose that $\sigma\in M^{\mathbb{P}}$, $\tau\in M^{\mathbb{Q}}$ with $\sigma^G=x$ and $\tau^H=x$. Then there are conditions $p\in \mathbb{P}$, $q\in \mathbb{Q}$ with $(p,q)\Vdash_{\mathbb{P}\times\mathbb{Q}} \sigma^{\dot{G}}=\tau^{\dot{H}}$. Suppose that $x\notin M$. Then for some $y\in M$, $p$ does not decide if $y\in \sigma^{\dot{G}}$, and hence $q$ does not decide if $y\in \tau^{\dot{H}}$. Suppose that $p'\leq p$, $q'\leq q$ with $p'\Vdash_{\mathbb{P}} y\in \sigma^{\dot{G}}$ and $q'\Vdash_{\mathbb{Q}} y\notin \tau^{\dot{H}}$. Then $(p',q')\Vdash_{\mathbb{P}\times\mathbb{Q}} \sigma^{\dot{G}}\neq \tau^{\dot{H}}$, contradicting the assumption that $(p,q)\Vdash_{\mathbb{P}\times\mathbb{Q}} \sigma^{\dot{G}}=\tau^{\dot{H}}$. 
\end{proof} 

\begin{lemma}{\label{comeagergen}}
Suppose that $\alpha\in\omega_1$ and $a\subseteq\omega$. Then the set $C_{\alpha}$ of Cohen-generic reals over $L_{\alpha}[a]$ is comeager.
\end{lemma}
\begin{proof}
Cohen forcing consists of functions of the form $p:n\rightarrow2$ for some $n\in\omega$. 
For $f:\omega\rightarrow2$, we define the filter $G_f$ as the set of all finite initial functions of $f$. Then $f$ is Cohen-generic iff $G_f$ is a Cohen-generic filter.
We show that the set of $f$ such that $G_f$ intersects every dense subset of $\mathbb{P}$ contained in $L_{\alpha}$ is comeager. 

To this end, we first demonstrate that, for a particular dense subset $D$ of $\mathbb{P}$, the set $N_{D}$ of $f$ such that $G_{f}\cap D=\emptyset$ is nowhere dense.
To see this, let $[x,y]$ be a non-empty interval. We have to show that there are $x^{\prime},y^{\prime}$ such that $x<x^{\prime}<y^{\prime}<y$ and such that 
$[x^{\prime},y^{\prime}]$ consists entirely of elements $h$ for which $G_{h}\cap D\neq\emptyset$. Pick a subinterval $I$ of $[x,y]$ of the form $[k2^{-m},(k+1)2^{-m}]$, where
$k,m\in\mathbb{N}$ and $k<2^m$. Thus $I$ consists of all reals having the binary presentation $b(k)$ of $k$ as an initial segment. As $D$ is dense, pick $d\in D$ such that
$b(k)\subseteq d$. Thus $d\in G_{h}\cap D$ for all $h$ such that $d\subseteq h$, so that $G_{h}\cap D\neq\emptyset$ for all such $h$. 
But the set of these $h$ clearly forms a subinterval of $I$, hence of $[x,y]$ and is hence as desired. 

Now, $L_{\alpha}[a]$ is countable and hence contains only countably many dense sets, say $(D_{i}|i\in\omega)$. The set $F$ of $f$ for which there is some $i\in\omega$ with $G_{f}\cap D_{i}=\emptyset$ is just
$\bigcup_{i\in\omega}N_{D_{i}}$. As we just saw that each $D_i$ is nowhere dense, it follows that $F$ is meager. Consequently, the complement of $F$, i.e. the set of $f$ such that $G_f$ intersects
every $D_i$, is comeager.
\end{proof}

\begin{lemma}{\label{mutgen}}
Suppose that $A\subseteq{}^{\omega}2$ is a nonmeager Borel set and $\alpha< \omega_1$. There are reals $x, y\in A$ such that $x$ is Cohen-generic over $L_{\alpha}$ and $y$ is Cohen-generic over $L_{\alpha}[x]$.
\end{lemma}
\begin{proof}
Let $C_{\alpha}$ denote the set of Cohen reals over $L_{\alpha}$. 
Then $C_{\alpha}$ is comeager and hence $A\cap C_{\alpha}$ is comeager. Suppose tha $x\in A\cap C_{\alpha}$ and let $C$ denote the set
of Cohen reals over $L_{\alpha}[x]$.  
Since $C$ is comeager, suppose that $y\in A\cap C$. Then $y$ is Cohen generic over $L_{\alpha}[x]$. 
Hence $x,y\in A$ are mutually Cohen generic over $L_{\alpha}$. 
\end{proof} 

\begin{lemma}{\label{delta0forcing}}
Let $\mathbb{P}$ denote Cohen forcing. Suppose that $L_{\alpha}$ is provident, $p\in\mathbb{P}$, $\vec{\sigma}\in L_{\alpha}$, $\varphi$ is a formula. Then 
\begin{enumerate} 
\item If $\varphi$ is a $\Delta_0$ formula, then $p\Vdash_{\mathbb{P}}^{L_{\alpha}} \varphi$ is $\Delta_1$ over $L_{\alpha}$. 
\item If $\varphi$ is a $\Sigma_n$ formula, then $p\Vdash_{\mathbb{P}}^{L_{\alpha}} \varphi$ is $\Sigma_n$ over $L_{\alpha}$. 
\item If $\varphi$ is a $\Pi_n$ formula, the $p\Vdash_{\mathbb{P}}^{L_{\alpha}} \varphi$ is $\Pi_n$ over $L_{\alpha}$. 
\end{enumerate} 
\end{lemma} 

\begin{proof} This is proved for $\Delta_0$ formulas in \cite[Section 3]{Ma}. The rest follows inductively from the definition of the forcing relation. 
\end{proof} 

\begin{lemma} Let $\mathbb{P}$ denote Cohen forcing. Suppose that $L_{\alpha}$ is provident, $p\in\mathbb{P}$, $\varphi$ is a formula, and $\vec{\sigma}\in L_{\alpha}$. Then 
\begin{enumerate} 
\item $p\Vdash \varphi(\vec{\sigma})$ if and only if $L_{\alpha}[G]\vDash \varphi(\vec{\sigma}^G)$ for all Cohen generic filters $G$ over $L_{\alpha+1}$. 
\item Suppose that $G$ is Cohen generic over $L_{\alpha+1}$. Then $L_{\alpha}[G]\vDash \varphi(\vec{\sigma})$ if and only if $p\Vdash_{\mathbb{P}} \varphi(\vec{\sigma})$ for some $p\in G$. 
\end{enumerate} 
\end{lemma} 

\begin{proof} This follows from the proof of the forcing theorem, see for example \cite[Theorems 3.5 and 3.6]{Ku}. 
\end{proof} 

The previous lemma shows that $L_{\alpha}[x]\prec_{\Sigma_n} L_{\beta}[x]$ for all $n\geq 1$ and provident sets $L_{\alpha}\subseteq L_{\beta}$. This immediately implies the following. 

\begin{lemma}{\label{gensubmodels}} 
Suppose that $x$ is Cohen generic over $L_{\Sigma+1}$. 
\begin{enumerate} 
\item  $L_{\lambda}[x]\prec_{\Sigma_{1}}L_{\zeta}[x]\prec_{\Sigma_{2}}L_{\Sigma}[x]$. 
\item $\lambda^{x}=\lambda$, $\zeta^{x}=\zeta$ and $\Sigma^{x}=\Sigma$.
\end{enumerate} 
\end{lemma}

\begin{prop}{\label{ITTMcomeager}}
Suppose that $x$ is a real and that $A$ is a comeager set of reals such that $x$ is writable (eventually writable, accidentally writable) in every oracle $y\in A$. Then $x$ is writable (eventually writable, accidentally writable). 
\end{prop} 

\begin{proof} 
The set $C$ of Cohen generic reals over $L_{\Sigma+1}$ is comeager by Lemma \ref{comeagergen}, 
so $A\cap C$ is comeager. 
We may assume without loss of generality that $A\subseteq C$. 
The reals writable in every $y\in A$ are those in $\bigcap_{y\in A}L_{\lambda}[y]$, the reals eventually writable in every $y\in A$ are those
in $\bigcap_{y\in A}L_{\zeta}[y]$, and the reals accidentally writable in every $y\in A$ are those in $\bigcap_{y\in A}L_{\Sigma}[y]$, by Lemma \ref{gensubmodels} and Theorem \ref{ITTMchar}. 

Since $A$ is comeager, $A$ contains two mutually Cohen generic reals $u$ and $v$ by Theorem \ref{mutgen}. Since $\lambda$, $\zeta$ and $\Sigma$ are limits of admissibles, it is readily
seen that $L_\lambda$, $L_\zeta$ and $L_\Sigma$ are provident. Then 
$$L_{\lambda}\subseteq\bigcap_{y\in A}L_{\lambda}[y]\subseteq L_{\lambda}[u]\cap L_{\lambda}[v]=L_{\lambda}$$
$$L_{\zeta}\subseteq\bigcap_{y\in A}L_{\zeta}[y]\subseteq L_{\zeta}[u]\cap L_{\zeta}[v]=L_{\zeta}$$ 
$$L_{\Sigma}\subseteq\bigcap_{y\in A}L_{\Sigma}[y]\subseteq L_{\Sigma}[u]\cap L_{\Sigma}[v]=L_{\Sigma}$$ 
by Theorem \ref{mutgensect}. 
Hence we have equalities in each case 
and the claim follows from Theorem \ref{ITTMchar}. 
\end{proof} 

%
%

\begin{thm}
Suppose that $x$ is a real and that $A$ is a nonmeager Borel set of reals such that $x$ is writable (eventually writable, accidentally writable) in every oracle $y\in A$. Then $x$ is writable (eventually writable, accidentally writable). 
\end{thm}
\begin{proof}
Since $A$ has the Baire property, there is some finite $t$ such that, for the corresponding basic open set $N_{t}:=\{x|t\subseteq x\}$, $A\triangle N_t$ is meager.
Consequently, $A\cap N_{t}$ is comeager in $N_{t}$. We define a translation function $t:[0,1]\rightarrow N_{t}$, where $t(x)$ is obtained from $x$ by replacing the sequence of the first $|t|$ many bits of $x$ with $t$. Then $range(f)=N_{t}$, and $X:=f^{-1}[A\cap N_{t}]$ is comeager in $[0,1]$. Furthermore, $t$ is clearly $ITTM$-computable.
Now, if some $y$ is writable in every $a\in A$, then it is writable in every $t(x)$ with $x\in X$. So we can compute $y$ from every element of $X$ by first applying $f$ and then the reduction from $N_{t}$ to $y$.
Hence $y$ is writable in all elements of a comeager set, so $y$ is writable by Theorem \ref{ITTMcomeager}. The same argument shows the analogous statement for eventual and accidental writability.
\end{proof}

\section{Infinite time register machines}

Before we consider infinite time register machines, let us briefly mention the unresetting version of these machines. 
Unresetting (or \emph{weak}) $ITRM$s \cite{wITRM}, 
also called weak $ITRM$s ($wITRM$s), work like classical register machines. In particular, they use finitely many registers each of which can store a single natural number, but with transfinite ordinal running time. At limit times, the program line is the limit inferior of the earlier program lines and there is a similar limit rule for the register contents. 
If the limit inferior is infinite, then the computation is undefined.
A real $x$ is $wITRM$-computable if and only if 
$x\in L_{\omega_{1}^{CK}}$ \cite{wITRM}, and the proof  relativizes. 

\begin{lemma}
A real $x$ is $wITRM$-computable in the oracle $y$ if and only if $x\in L_{\omega_{1}^{CK,y}}[y]$.
\end{lemma}

Hence the question is whether there is a set $A$ of positive measure and a real $x\notin L_{\omega_{1}^{CK}}$ such that $x\in L_{\omega_{1}^{CK,y}}[y]$. 
We will use the following result (see \cite[Theorem 9.1.13]{Nies}), where $\leq_{h}$ denotes hyperarithmetic reducibility. 

\begin{thm}[Sacks]{\label{Nies}}
Suppose that $x$ is a real. Then $x\notin\Delta_{1}^{1}$ if and only if $x\notin L_{\omega_{1}^{CK}}$ if and only if $\mu(\{a|x\leq_{h}a\})=0$. 
\end{thm}

\begin{thm}{\label{wITRM}} 
Suppose that $x$ is a real and $A$ is a set of reals with $\mu(A)>0$ such that $x$ is $wITRM$-computable from every $y\in A$. Then $x$ is $wITRM$-computable. 
\end{thm} 
\begin{proof} 
Since $\mu(\{y|\omega_{1}^{CK,y}=\omega_{1}^{CK}\})=1$, we may assume that $\omega_{1}^{CK,y}=\omega_{1}^{CK}$ and thus $L_{\omega_{1}^{CK,y}}[y]=L_{\omega_{1}^{CK}}[y]$ for all $y\in A$. 
If $y$ is not $wITRM$-computable, then $y$ is not hyperarithmetical \cite{wITRM}. Then $\mu(\{x|y\leq_{h}x\})=0$ by Theorem \ref{Nies}, 
contradicting the assumption $\mu(A)>0$. 
\end{proof}


For the rest of this section, we consider (resetting) infinite time register machines. They
differ from weak $ITRM$s only in their behaviour when the limit inferior is infinite. 
In this case, the register in question is assigned the value 0 and the computation continues. 
This leads to a huge increase in terms of computability strength. An introduction to $ITRM$s can be found in \cite{ITRM}.

A real $x$ is $ITRM$-computable if and only if $x\in L_{\omega_{\omega}^{CK}}$ \cite{OC}.
$x\in L_{\omega_{\omega}^{CK}}$ and the proof relativizes. 

\begin{lemma}
A real $x$ 
is $ITRM$-computable in a real $y$ if and only if 
$x\in L_{\omega_{\omega}^{CK,y}}[y]$. 
\end{lemma}

The question is now whether 
there is a real $x\notin L_{\omega_{\omega}^{CK}}$ and a set $A$ of positive measure 
such that $x\in L_{\omega_{\omega}^{CK,y}}[y]$ for every $y\in A$. 
To show that there is no such real, we first relativize Theorem \ref{Nies}. 

\begin{prop}{\label{RelNies}} 
Suppose that $x,y$ are reals. 
Then $x\notin L_{\omega_{1}^{CK,y}}[y]$ if and only if $\mu(\{a\mid x\leq_{h} a\oplus  y\})=0$. 
\end{prop} 
\begin{proof} 
We follow the proof of \cite[Theorem 3.1.1]{Ke}. 
Suppose that $x\notin L_{\omega_{1}^{CK,y}}[y]$. The set $\{a\mid x\leq_{h} a\oplus  y\}$ is $\Pi^1_1$ in $y$. Let us assume that it has positive measure. 
Since there are only countably many hyperarithmetic reductions, there is some hyperarithmetic reduction $P$ such that for a positive measure set of $a$, $P$ reduces $x$ to $a\oplus y$.
 Then there is a rational interval $I$ in which this set has relative measure $>0.5$ by the Lebesgue density theorem. 
The set $\{b\in I\mid P^{a\oplus y}=P^{b\oplus y}\}$ is $\Pi^1_1$ in $a$ and hence measurable. 
We define $Y$ as the set of $a$ with $\mu_I(\{b\in I\mid P^{a\oplus y}=P^{b\oplus y}\})> 0.5$, where $\mu_I(A)=\frac{\mu(A\cap I)}{\mu(I)}$ denotes the relative measure. Then $Y$ is $\Pi^1_1$ in $y$ by \cite[Theorem 2.2.3]{Ke}. 
The set $Z:=\{z\in {}^{\omega}2\mid \omega_1^{CK,z}=\omega_1^{CK}\}$ has measure $1$ by \cite[Corollary 9.1.15]{Nies}. 
Since $\mu(Y)>0.5$ there is some $z\in Y$ with $\omega_1^{CK,z}=\omega_1^{CK}$. 
Since $Y$ is $\Pi^1_1$ in $y$, there is a tree $T$ is a tree computable in $y$ such that $z\in Y$ if and only if $T_z$ is wellfounded, for all $z\in {}^{\omega}2$. 
Since $Z$ has measure $1$, there is some $z\in Y\cap Z$. 
Then $\alpha:=\mathrm{rank}(T_z)<\omega_1^{CK,z}=\omega_1^{CK}$ by \cite[Theorem 4.4]{Hj}. 
Since $\alpha$ is computable, the set $X:=\{z\in {}^{\omega}2\mid \mathrm{rank}(T_z)\leq\alpha\}$ is a nonempty $\Delta^1_1$ in $y$ subset of $Y$. 
Then $P^{a\oplus y}=x$ for all $a\in Y$. 
Then $z=x$ if and only if $P^{u\oplus y}=z$ for some (for all) $u\in X$. Since $P^{u\oplus y}$ halts 
for all $u\in X$, $P^{u\oplus y}=z$ can be equivalently replaced by the statement that for every halting run of $P$ on input $u\oplus y$ the output is $z$. Thus $x$ is hyperarithmetic in $y$. 
\end{proof}

\begin{lemma}{\label{RelAdm}} Suppose that $x$ is a real. 
\begin{enumerate} 
\item $\mu(\{y\in {}^{\omega}2\mid \omega_{1}^{CK,x\oplus y}=\omega_{1}^{CK,x}\})=1$. 
\item $\mu(\{y\in{}^{\omega}2 \mid \forall{i\in\omega}\ \omega_{i}^{CK,y}=\omega_{i}^{CK}\})=1$. 
\item $\mu(\{y\in{}^{\omega}2 \mid \omega_{\omega}^{CK,y}=\omega_{\omega}^{CK}\})=1$. 
\end{enumerate} 
\end{lemma}
\begin{proof}
1. Let $c(x)$ denote the $<_{L[x]}$-least real $r$ which codes a well-ordering of length $\omega_{1}^{CK,x}$. Now suppose that $y$ is such that $\omega_{1}^{CK,x\oplus y}>\omega_{1}^{CK,x}$.
Then $x\in L_{\omega_{1}^{CK,x}}[x]\in L_{\omega_{1}^{CK,x\oplus y}}[x\oplus y]$ and $L_{\omega_{1}^{CK,x\oplus y}}[x]\subseteq L_{\omega_{1}^{CK,x\oplus y}}[x\oplus y]$. 
Let $H$ denote the hull of $x$ in $L_{\omega_{1}^{CK,x}}[x]$ for the canonical Skolem functions. 
Then $H=L_{\omega_{1}^{CK,x}}[x]$ by condensation \cite[Theorem 1.16]{SZ} and since $L_{\omega_{1}^{CK,x}}[x]$ is the least model of $KP$ containing $x$. 
Then there is a bijection between $\omega$ and $\omega_{1}^{CK,x}$ and hence a code $c$ for $\omega_{1}^{CK,x}$ in $L_{\omega_{1}^{CK,x}+\omega}[x]$. 
As $c(x)\leq_{L[x]}c$ by the minimality of $c(x)$, we 
have $c(x)\in L_{\omega_{1}^{CK,x\oplus y}}[x\oplus y]$ and $c(x)\leq_{h}x\oplus y$. 
Moreover $c(x)\not\leq_{h}x$ implies $\mu(\{y\mid c(x)\leq_{h}x\oplus y\})=0$ by Theorem \ref{RelNies}. 

2. Let $c(i)$ denote the $<_{L}$-least code for $\omega_{i}^{CK}$. We have $\mu(\{x\in{}^{\omega}2\mid \omega_{1}^{CK,x\oplus y}=\omega_{1}^{CK,y}\})=1$ for all $y\in{}^{\omega}2$ by Lemma \ref{RelAdm}. 
Then 
$$X_{i}:=\{x\in{}^{\omega}2 \mid  \omega_{1}^{CK,x\oplus c(i)}=\omega_{1}^{CK,c(i)}\}$$ has measure $1$ for all $i\in\omega$ and hence $X=\bigcap_{i\in\omega}X_{i}$ has measure $1$. 
We claim that $\omega_{i}^{CK,x}=\omega_{i}^{CK}$ for all $x\in X$. To see this, let us denote by $c(i,y)$ the $<_{L}$-least code for $\omega_{i}^{CK,y}$ for
$y\in{}^{\omega}2$. Then $c(0,y)$ is a code for $\omega$ and $\omega_{1}^{CK,y\oplus c(i,y)}=\omega_{i+1}^{y}$ for all reals $y$. 
Now suppose that $x\in X$. Since $x\in X_{1}$, we have 
$\omega_{1}^{CK,x}=\omega_{1}^{CK,x\oplus c(0,x)}=\omega_{1}^{CK,c(0)}=\omega_{1}^{CK}$. 
If $\omega_{i}^{CK,x}=\omega_{i}^{CK}$, then $c(i,x)=c(i)$. 
Since $x\in X_{i+1}$, we have
$\omega_{i+1}^{CK,x}=\omega_{1}^{CK,c(i,x)\oplus x}=\omega_{1}^{CK,c(i)\oplus x}=\omega_{1}^{CK,x\oplus c(i)}=\omega_{1}^{CK,c(i)}=\omega_{i+1}^{CK}$. 
Hence $\omega_{i}^{CK,x}=\omega_{i}^{CK}$ for all $i\in\omega$. 

3. This follows from the previous claim, since $\omega_{\omega}^{CK,x}=\sup_{i\in\omega} \omega_i^{CK,x}$. 
\end{proof}


We can now show that $ITRM$-computability relative to oracles in a set of positive measure implies $ITRM$-computability.

\begin{thm}{\label{ITRMrandom}}
Suppose that $x$ is a real 
and $A$ is a set of positive measure such that $x$ is $ITRM$-computable from all $y\in A$. Then $x$ is $ITRM$-computable. 
\end{thm}
\begin{proof} 
It is sufficient to show that $\bigcap_{y\in A}L_{\omega_{\omega}^{CK,y}}[y]=L_{\omega_{\omega}^{CK}}$.
Suppose that $x\in\bigcap_{y\in A}L_{\omega_{\omega}^{CK,y}}[y]\setminus L_{\omega_{\omega}^{CK}}$. 
Then for each $y\in A$, there is a least $i(y)\geq 1$ with $x\in L_{\omega_{i}^{CK,y}}[y]$. Let $A_{j}:=\{y\in A\mid i(y)=j\}$ for $j\in\omega$. 
Then $A=\bigcup_{j\in\omega}A_{j}$ and since the sets $A_j$ are provably $\Delta^1_2$, they are measurable by \cite[Exercise 14.4]{Kanamori}. 
Hence $\mu(A_k)>0$ for some $k\geq 1$. 
If $k=1$, then $x\in L_{\omega_1^{CK,y}}[y]$ for all $y\in A_1$ and $\mu(A_1)>0$, so $x$ is $ITRM$-computable. 
Suppose that $k=j+1$. Let $c$ denote the $<_{L}$-least code for a wellorder of length 
$\omega_{j}^{CK}$. 
Then there is a $\Sigma_1$ over $L_{\omega_{j}^{CK,y}}[y]$ 
partial surjection of $\omega$ onto $\omega_{j}^{CK}$ for all $y\in A_k$ and 
hence $c\in L_{\omega_{j}^{CK,y}+1}[y]$. 
Then $x\in L_{\omega_{k}^{CK,y}}[y]=L_{\omega_{1}^{CK,c\oplus y}}[c\oplus y]$ 
and hence $x\leq_{h} c\oplus y$ for $y\in A_{k}$. 
Then $x\in L _{\omega_{1}^{CK,c}}[c]=L_{\omega_{k}^{CK}}\subseteq L_{\omega_{\omega}^{CK}}$ by Theorem \ref{RelNies}, since $\mu(A_k)>0$. 
\end{proof}

Let us call a real $x$ \emph{$ITRM$-extracting} if and only if there is a real $y$ which is not $ITRM$-computable from $x$, but the set of reals $z$ such that $y$ is $ITRM$-computable from $x\oplus z$ has positive measure. 
A slight generalization of the above ideas shows that there are also no extracting reals for $ITRM$s, in contrast to the case of $OTM$s, where this is independent from $ZFC$. 

\begin{lemma}{\label{intersect}} 
Suppose that $x,y$ are reals and $\omega_{j}^{CK}=\omega_{j}^{CK,y}$ for all $j\in\omega$. 
Suppose that $i\in\omega$ and $c(i)$ is the $<_L$-least code for $\omega_{i}^{CK}$. 
Then $x\notin L_{\omega_{i+1}^{CK,y}}[y]$ if and only if $\mu(\{z\mid x\leq_{h} z\oplus y\oplus c(i)\})=0$. 
\end{lemma}
\begin{proof}
Since $\omega_{j}^{CK}=\omega_{j}^{CK,y}$ for all $j\in\omega$, 
we have $\omega_{i+1}^{CK,y}=\omega_{1}^{CK,c(i)\oplus y}$. 
Then $x\in L_{\omega_{i+1}^{CK,y}}[y]\subseteq L_{\omega_{1}^{CK,y\oplus c(i)}}[y\oplus c(i)]$ implies that $x\leq_{h} y\oplus c(i)$.
For the other direction, suppose that $x\notin L_{\omega_{i+1}^{CK,y}}[y]=L_{\omega_{1}^{CK,y\oplus c(i)}}[y\oplus c(i)]$. Then $\{z\mid x\leq_{h} z\oplus y\oplus c(i)\}=\{z\mid x\leq_{h}z\oplus(y\oplus c(i))\}$ has 
measure $0$ by Theorem \ref{RelNies} applied to $y\oplus c(i)$. 
\end{proof}

\begin{prop}{\label{noITRMextractings}}
There is no $ITRM$-extracting real.
\end{prop}
\begin{proof} Assume for a contradiction that $x$ is $ITRM$-extracting, witnessed by a real $y$. Then $y\notin L_{\omega_{\omega}^{CK,x}}[x]$ and $y\in L_{\omega_{\omega}^{CK,x\oplus z}}[x\oplus z]$ for a set of reals $z$ of positive measure. 
We have $y\nleq_{h}c(i)\oplus x$ if and only if $\mu(\{z\mid y\leq_{h}z\oplus x\oplus c(i)\})=0$ for all $i\in\omega$ by Lemma \ref{intersect}. Hence 
$$
\begin{array}{lll} 
   y\notin L_{\omega_{\omega}^{CK,x}}[x] & \Leftrightarrow & \forall{i\in\omega}\ y\notin L_{\omega_{i}^{CK,x}}[x] \nonumber \\
   & \Leftrightarrow & \forall{i\in\omega}\ y\nleq_{h}c(i)\oplus x \\ 
   & \Leftrightarrow & \forall{i\in\omega}\ \mu(\{z\mid y\leq_{h}c(i)\oplus z\oplus x\})=0 \\
   & \Leftrightarrow & \forall{i\in\omega}\ \mu(\{z\mid y\in L_{\omega_{i}^{CK,z\oplus x}}[z\oplus x]\})=0 \\ 
   & \Leftrightarrow & \mu(\{z\mid y\in L_{\omega_{\omega}^{z\oplus x}}[z\oplus x]\})=0 \\ 
\end{array}
$$
contradicting the assumption on $y$. 
\end{proof}

\begin{remark} A similar strategy works for the other machine types considered in this paper besides $OTM$s and $ORM$s and the arguments relativize in a straightforward manner. 
\end{remark}

We now prove an analogous result for nonmeager Borel sets of oracles. 



\begin{lemma}{\label{cohenadmlimpres}} 
\begin{enumerate} 
\item If $g$ is Cohen generic over $L_{\omega_{\omega}^{CK}}$, then $\omega_{\omega}^{CK,g}=\omega_{\omega}^{CK}$. 
\item If $g$ is Cohen generic over $L_{\omega_{i}^{CK}+1}$, then $\omega_{i}^{CK,g}=\omega_{i}^{CK}$. 
\end{enumerate} 
\end{lemma} 
\begin{proof} 
1. If $\alpha$ is admissible and $h$ is a Cohen generic filter over $L_{\alpha+1}$, then $L_{\alpha}[h]$ is admissible by Theorem $10.1$ of \cite{Ma}. 
Note that $g$ is Cohen generic over $L_{\omega_{i}^{CK}+1}$ for all $i\in\omega$. 
Then $\omega_i^{CK,g}=\omega_i^{CK}$ for all $i\in\omega$. 
Hence $\omega_{\omega}^{CK,g}=\bigcup_{i\in\omega}\omega_{i}^{CK,g}=\bigcup_{i\in\omega}\omega_{i}^{CK}=\omega_{\omega}^{CK}$. 

2.  
As in the proof of the previous claim, 
$\omega_{j}^{CK}$ is $g$-admissible for all $j\leq i$, so that $\omega_{j}^{CK,g}=\omega_{j}^{CK}$ for all $j\leq i$. 
\end{proof}


\begin{thm}{\label{ITRMmeager}} 
Suppose that $x$ is a real and $A$ is a nonmeager Borel set such that $x$ is $ITRM$-computable from all $y\in A$. 
Then $x$ is $ITRM$-computable.
\end{thm}
\begin{proof} We can assume that there is some $ITRM$ program $P$ which computes $x$ from all $y\in A$. 
The set $C$ of Cohen reals over $L_{\omega_{\omega}^{CK}}$ is comeager, so we can assume that $A\subseteq C$. 
There are mutually Cohen generic reals $u,v\in A$ over 
$L_{\omega_{\omega}^{CK}}$ by Lemma \ref{mutgen}. 
Then $L_{\omega_{\omega}^{CK,u}}[u]\cap L_{\omega_{\omega}^{CK,v}}[v]=L_{\omega_{\omega}^{CK}}[u]\cap L_{\omega_{\omega}^{CK}}[v]=L_{\omega_{\omega}^{CK}}$ by Lemma \ref{mutgensect}. 
Then 
$$L_{\omega_{\omega}^{CK}}\subseteq\bigcap_{y\in A}L_{\omega_{\omega}^{CK}}[y]
\subseteq L_{\omega_{\omega}^{CK}}[u]\cap L_{\omega_{\omega}^{CK}}[v]=L_{\omega_{\omega}^{CK}}$$ 
and hence $x$ is $ITRM$-computable. 
\end{proof}

\begin{remark} 
Following the same line of reasoning, 
if $x$ is $wITRM$-computable from all oracles in a nonmeager Borel set $A$ of oracles, then $x$ is $wITRM$-computable. 
\end{remark}

\section{$\alpha$-Turing machines}

Suppose that $\alpha>\omega$ is a countable admissible ordinal. 
In this section, we consider computability relative to a set of oracles of positive measure for parameter free $\alpha$-Turing machines as defined in \cite{KoSe}. These machines are similar to ITTMs, but have tape length $\alpha$. 
We crucially use the following characterization of the computability strength of $\alpha$-Turing machines. 
This is a minor modification of \cite[Lemma 3]{KoSe}. 

\begin{lemma}{\label{alphacomp}} Suppose that $\alpha>\omega$ is exponentially closed. 
A real $x$ is computable by an $\alpha$-Turing machine in an oracle $y$ if and only if $x$ is $\Delta_{1}$-definable in the parameter $y$ over $L_{\alpha}[y]$. 
\end{lemma} 

In particular, for reals $x,y$ with $\omega_{i}^{CK,y}=\omega_{i}^{CK}$, $x$ is $\omega_{i}^{CK}$-computable from $y$ if and only if $x\in L_{\omega_{i}^{CK}}[y]$. 

If $\alpha$ is an ordinal, let $\alpha^+$ denote the least admissible ordinal $\gamma>\alpha$. Let $\bar{\alpha}=\omega_{\bar{\iota}}$ denote the least admissible ordinal $\gamma$ such that $L_{\gamma^{+}}$ does not contain a real coding $\gamma$. 
Then for every admissible $\alpha<\bar{\alpha}$, the $<_L$-least real $c_{\alpha}$ coding $\alpha$ is in $L_{\alpha^+}$. 
We will extend the preceding results to all admissible ordinals $\alpha<\bar{\alpha}$. 


\begin{lemma}{\label{invariantadmissibles}} 
 If $\iota<\bar{\iota}$, then $\mu(\{x\in{}^{\omega}2\mid \omega_{\iota}^{CK,x}=\omega_{\iota}^{CK}\})=1$. 
\end{lemma} 
\begin{proof} 
The proof is similar to Lemma \ref{RelAdm}, where the case $\iota<\omega$ was proved. 
Suppose that $\iota<\bar{\alpha}$ and the claim is known for all $\gamma<\iota$. 
Let $M_{\gamma}:=\{y\mid \omega_{\gamma}^{CK,y}=\omega_{\gamma}^{CK}\}$ for $\gamma<\iota$ 
and $M:=\bigcap_{\delta<\iota}M_{\delta}$. Then $\mu(M)=1$. 
If $\iota=\gamma+1$, then 
$\mu(\{z\mid \omega_{1}^{CK,z\oplus c_{\gamma}}=\omega_{1}^{CK,c_{\gamma}}\})=1$ by Lemma \ref{RelAdm}. 
Since $\omega_{1}^{CK,c_{\gamma}}=\omega_{\gamma+1}^{CK}=\omega_{\iota}^{CK}$, this implies $\mu(\{y\in M\mid \omega_{1}^{CK,y\oplus c_{\gamma}}=\omega_{\iota}^{CK}\})=1$. 
For all $y\in M$, 
we have $\omega_{\gamma}^{CK,y}=\omega_{\gamma}^{CK}$ and $\omega_{1}^{CK,y\oplus c_{\gamma}}=\omega_{\gamma+1}^{CK,y}=\omega_{\iota}^{CK}$, so $\omega_{\iota}^{CK,y}=\omega_{\iota}^{CK}$. 
If $\iota$ is a limit ordinal, then  $\omega_{\gamma}^{CK,y}=\omega_{\gamma}^{CK}$ for all $\gamma<\iota$ and $y\in M$. 
Then $\omega_{\iota}^{CK,y}=\bigcup_{\gamma<\iota}\omega_{\gamma}^{CK,y}=\bigcup_{\gamma<\iota}\omega_{\gamma}^{CK}=\omega_{\iota}^{CK}$ for all $y\in M$ and 
$\mu(\{x\in{}^{\omega}2\mid \omega_{\iota}^{CK,x}=\omega_{\iota}^{CK}\})=1$. 
\end{proof}

Consequently
$L_{\omega_{\iota}^{CK,x}}[x]=L_{\omega_{\iota}^{CK}}[x]$ for almost all $x$ and all $\iota<\bar{\iota}$. 

\begin{thm}{\label{strongversionalpha}}
Suppose that $\alpha=\omega_{\iota}^{CK}<\bar{\alpha}$ is admissible, $x$ is a real, $A$ is a set of positive measure, and $P$ is an $\alpha$-Turing program such that $P^{y}=x$ for all $y\in A$. Then $x$ is $\alpha$-computable. 
\end{thm}
\begin{proof}
Suppose that $\iota<\bar{\iota}$ and that the claim holds for all $\gamma<\iota$. 
We have $x\in L_{\omega_{\iota}^{CK,y}}[y]$ for all $y\in A$, so we can assume that $\omega_{\gamma}^{CK,y}=\omega_{\gamma}^{CK}$ for all $\gamma\leq\iota$ by Lemma \ref{invariantadmissibles}. 
Then $x\in L_{\omega_{\iota}^{CK}}[y]$ for all $y\in A$. 

If $\iota=1$, then we can assume that $\omega_{1}^{CK,y}=\omega_{1}^{CK}$ for all $y\in A$, since $\mu(\{y\subseteq\omega\mid \omega_{1}^{CK,y}=\omega_{1}^{CK}\})=1$ by Lemma \ref{RelAdm}. 
Then $x\in\bigcap_{y\in A}L_{\omega_{1}^{CK}}[y]$ by Lemma \ref{alphacomp}, so $x\leq_{h}y$ for all $y\in A$. Then $x\in L_{\omega_{1}^{CK}}$ by Theorem \ref{Nies} and hence $x$ is $\omega_{1}^{CK}$-computable by Lemma \ref{alphacomp}. 

If $\iota=\gamma+1>1$, then $x\in L_{\omega_{1}^{CK,c_{\gamma}\oplus y}}[c_{\gamma}\oplus y]=L_{\omega_{\iota}^{CK}}[y]$ and hence $x\leq_{h} c_{\gamma}\oplus y$ for all $y\in A$. 
If $x\leq_{h}c_{\gamma}$ then 
$x\in L_{\omega_{1}^{CK,c_{\gamma}}}[c_{\gamma}]=L_{\omega_{\iota}^{CK}}=L_{\alpha}$ and $x$ is $\alpha$-computable, as desired. 
If $x\nleq_{h}c_{\gamma}$ then $\mu(\{z|x\leq_{h}z\oplus c_{\gamma}\})=0$ by Lemma \ref{RelNies}. 
Since $x\leq c_{\gamma}\oplus a$ for all $a\in A$, 
this implies $\mu(A)=0$,  contradicting the assumption on $A$. 
 
If $\iota$ is a limit ordinal, then $x\in\bigcup_{\gamma<\iota}L_{\omega_{\gamma}^{CK}}[y]$ for all $y\in A$. There are $\gamma_y<\iota$ with $x\in L_{\omega_{\gamma_y}^{CK}}[y]$ for all $y\in A$. 
Let $A_{\gamma}:=\{y\in A|\gamma_y=\gamma\}$ for $\gamma<\iota$. 
Since $A_{\gamma}$ is provably $\Delta_{2}^{1}$, it is measurable \cite[Exercise 14.4]{Kanamori}. 
Then $\mu(A_{\gamma})>0$ for some $\gamma<\iota$. 
Hence $x\in L_{\omega_{\gamma}^{CK}}[y]$ for all $y\in A_{\gamma}$ and $x\in L_{\omega_{\gamma}^{CK}}\subseteq L_{\alpha}$. 
\end{proof} 

This can be extended to unboundedly many countable admissibles. 

\begin{thm}
There unboundedly many countable admissible ordinals $\alpha$ 
such that every real $x$ which is $\alpha$-computable from all elements of a set $A$ of positive measure is $\alpha$-computable. 
\end{thm}
\begin{proof} Suppose that $T$ is a finite fragment of $ZFC$ which is sufficient for the proof of Lemma \ref{successive random reals}. Then $L_{\alpha}\vDash T$ for unboundedly many countable admissible ordinals. 
Suppose that $L_{\alpha}\vDash T$. 
Since $\mu(A)>0$, there are $y,z\in A$ such that $y$ is random generic over $L_{\alpha}$ and $z$ is random generic over $L_{\alpha}[y]$. 
Then $L_{\alpha}[y]\cap L_{\alpha}[z]=L_{\alpha}$ by Lemma \ref{successive random reals} and hence $x\in L_{\alpha}$. 
\end{proof}

%
%
%

Let us calculate bounds on $\bar{\alpha}$. Let $\alpha_0$ denote the least $\beta$ such that $L_{\alpha}$ is elementarily equivalent to $L_{\beta}$ for some $\alpha<\beta$. Recall the $\eta$ denotes the supremum of the halting times of $OTM$s. 

\begin{lemma} $\alpha_0\leq \bar{\alpha}<\eta$. 
\end{lemma}
\begin{proof} 
Suppose that $\gamma<\alpha_0$ 
is admissible. 
To see that $L_{\gamma^{+}}$ contains a real coding $L_{\gamma}$, 
let $S$ denote the set of all sentences which hold in $(L_{\gamma},\in)$. 
Since $\gamma<\alpha_0$, $L_{\gamma}$ is minimal such that $L_{\gamma}\models S$. 
Let $H$ 
denote the elementary hull of the empty set in $L_{\gamma}$ with respect to the canonical Skolem functions and let 
$L_{\bar{\gamma}}$ denote the transitive collapse of $H$. 
Then $H=L_{\bar{\gamma}}=L_{\gamma}$ by the minimality of $\gamma$. 
Then there is a surjection from $\omega$ onto $H$ and a real coding $L_{\gamma}$ in $L_{\gamma^+}$. 

To see that $\bar{\alpha}<\eta$, recall that $\eta$ is the supremum of the $\Sigma_{1}$-fixed ordinals by Lemma \ref{OTMhaltingtimesup}. 
The existence of admissibles $\alpha<\beta$ such that there is a real $x\in L_{\beta}\setminus L_{\alpha}$ is expressed by a $\Sigma_1$ formula which first becomes true in some $L_{\gamma}$ with $\gamma>\bar{\alpha}$. 
This implies $\eta>\bar{\alpha}$. 
\end{proof}


A generalization of the argument for $ITRM$s shows an analogous result for $\alpha$-Turing machines for admissible ordinals $\alpha$ and nonmeager Borel sets of oracles. 



\begin{thm}{\label{alphaTMmeager}} 
Suppose that $x$ is a real, $\alpha$ is a countable admissible ordinal, $A$ is a nonmeager Borel set of reals, and $P$ is an $\alpha$-Turing program such that $P^{y}=x$ for all $y\in A$. Then $x$ is $\alpha$-computable. 
\end{thm} 
\begin{proof} Suppose that $\alpha=\omega_{\iota}^{CK}$. 
If 
$x$ is Cohen generic over $L_{\alpha+1}$, then $\omega_{\iota}^{CK,x}=\omega_{\iota}^{CK}$. This follows from the fact that for admissible $\beta<\alpha$, $L_{\beta}[x]$ is admissible by \cite[Theorem 10.1]{Ma}. 
The set $C$ of Cohen generic reals over $L_{\alpha+1}$ is comeager by Lemma \ref{comeagergen}, so we can assume that $A\subseteq C$. 
Then $\omega_{\iota}^{CK,y}=\omega_{\iota}^{CK}$ for all $y\in A$. 
There are mutual Cohen generics $u,v\in A$ over $L_{\alpha+1}$ by Lemma \ref{mutgen}. 
Then 
$$L_{\omega_{\iota}^{CK,u}}[u]\cap L_{\omega_{\iota}^{CK,v}}[v]=L_{\omega_{\iota}^{CK}}[u]\cap L_{\omega_{\iota}^{CK}}[v]=L_{\omega_{\iota}^{CK}}=L_{\alpha}$$ by Lemma \ref{mutgensect}. 
Hence $x\in L_{\alpha}$ is $\alpha$-computable. 
\end{proof}

\section{Conclusion} 


We considered the question whether computability from all oracles in a set of positive measure implies computability for various machine models. For most models this is the case, and for $OTM$s under the additional assumption that all for all $x\in {}^{\omega}2$, the set of random reals over $L[x]$ has measure $1$. 
Thus these machine models share the intuitive property of Turing machines that no information can be extracted from random information. 


\begin{question} Suppose that $\alpha\leq \beta<\omega_1$ are admissible. Are there analogous results for $(\alpha,\beta)$-Turing machines with tape length $\alpha$ and running time bounded by $\beta$? 
\end{question} 

Analogous results fail for other natural notions of largeness even in the computable setting, for example for Sacks measurability. For any $x\in {}^{\omega}2$, there is a perfect tree $T\subseteq {}^{<\omega}2$ such that $x$ is computable from every branch $y\in [T]$. Moreover $[T]$ is Sacks measurable and not Sacks null (see \cite[Definition 2.6]{Ik}). 

\begin{question} Suppose that $A$ is Borel and ${}^{\omega}2\setminus A$ is Sacks null. If $x\in {}^{\omega}2$ is computable from every $y\in A$, is $x$ computable? 
\end{question} 

Various machine types correspond in a natural manner to variants of Martin-L\"of-randomness. A fascinating subject is how far the analogy goes in each case. In particular, for which machine types is it true that, if
$x$ is computable from two mutually $ML$-random reals $y$ and $z$, then must $x$ be computable? We are pursuing this in ongoing work.

\bibliographystyle{alpha}

\end{document}